\documentclass[12pt,letterpaper]{article}

\usepackage{amsfonts,amssymb,latexsym,amsthm, amsmath,color,tikz}

\newcommand{\overbar}[1]{\mkern 1.5mu\overline{\mkern-1.5mu#1\mkern-1.5mu}\mkern 1.5mu}
\newcommand\res{ \mathord{\upharpoonright}}
\DeclareMathOperator{\dom}{dom} \DeclareMathOperator{\rng}{rng}

\def\taui{\tau_{\scriptscriptstyle \mathcal{I}}}
\def\seq{{\omega}^{<\omega}}

\def\upn{\mathord{\uparrow}}

\newtheorem{theorem}{Theorem}[section]
\newtheorem{lemma}[theorem]{Lemma}
\newtheorem{corollary}[theorem]{Corollary}
\newtheorem{question}[theorem]{Question}

\theoremstyle{definition}
\newtheorem{definition}[theorem]{Definition}
\newtheorem{remark}[theorem]{Remark}
\newtheorem{example}[theorem]{Example}

\begin{document}

\author{Ramiro de la Vega\thanks{Universidad de los Andes, Bogot\'a, Colombia, rade@uniandes.edu.co}, Javier Murgas\thanks{javier\_murgas@hotmail.com}, Carlos Uzc\'ategui\thanks{Escuela de Matem\'aticas, Universidad Industrial de Santander, Bucaramanga, Colombia, cuzcatea@saber.uis.edu.co}}

\title{Selective separability and $q^+$ on maximal spaces}

\maketitle

\begin{abstract}
Given a hereditarily meager ideal $\mathcal{I}$ on a countable set
$X$ we use Martin's axiom for countable posets to produce a
zero-dimensional maximal topology $\tau^\mathcal{I}$ on $X$ such
that $\tau^\mathcal{I}\cap \mathcal{I}=\{\emptyset\}$ and,
moreover, if $\mathcal{I}$ is $p^+$ then $\tau^\mathcal{I}$ is
selectively separable (SS) and if $\mathcal{I}$ is $q^+$, so is
$\tau^\mathcal{I}$. In particular,  we obtain regular maximal
spaces satisfying all boolean combinations of the properties SS
and  $q^+$.
\end{abstract}



\section{Introduction}

A topological space $X$ is {\em selectively separable} (SS), if
for any sequence $(D_n)_n$ of dense subsets of $X$ there are
finite sets $F_n\subseteq D_n$,  for $n\in\omega$,  such that
$\bigcup_n F_n$ is dense in $X$. This notion was introduced by
Scheepers \cite{Scheeper99}  and  has received a lot of attention
ever since (see for instance \cite{BarmanDow2011, BarmanDow2012,
Bella2009, Bella_et_al2008, Bella2013, CamargoUzca2018b,
Gruenhage2011, Reposvetal2010}). Bella et al.
\cite{Bella_et_al2008}  showed that every separable space with
countable fan tightness is SS. On the other hand, Barman and Dow
\cite{BarmanDow2011}  showed that every separable Fr\'echet space
is also SS (see also \cite{CamargoUzca2018b}).

A topological space is  {\em maximal}  if it is a dense-in-itself
space such that any strictly finer topology has an isolated point.
It was shown by van Douwen \cite{Vand} that a space is maximal if,
and only if, it is {\em extremely disconnected} (i.e. the closure
of every open set is open), {\em nodec} (i.e. every nowhere dense
set is closed)  and every open set is {\em irresolvable} (i.e. if
$U$ is open and $D\subseteq U$ is dense in $U$, then $U\setminus
D$ not dense in $U$). He constructed a countable maximal regular
space.

A countable space $X$  is $q^+$ at a point $x\in X$, if given any
collection of finite sets $F_n\subseteq X$ such that $x\in
\overline{\bigcup_n F_n}$, there is $S\subseteq \bigcup_n F_n$
such that  $x\in \overline{S}$ and  $ S\cap F_n$ has at most one
point  for each $n$. We say that $X$ is a {\em $q^+$-space} if it
is $q^+$ at every point.  Every countable sequential space is
$q^+$ (see \cite[Proposition 3.3]{Todoruzca2000}).  This notion is
motivated by the analogous concept of a $q^+$ filter  used  in
Ramsey theory.

The existence of a maximal regular SS space  is independent of
ZFC. In fact, in ZFC  there is a maximal non SS space
\cite{BarmanDow2011} and  it is consistent with ZFC that no
countable maximal space is SS \cite{BarmanDow2011,
Reposvetal2010}. On the other hand, using MA for countable posets,
Barman and Dow  \cite{BarmanDow2011} showed that  it is also
consistent that there is a maximal,  SS regular countable space.

Similar questions have been studied in the context of  countable
spaces with an analytic topology (i.e. the topology of the space
$X$ is an analytic subset of $2^X$
\cite{todoruzca,Todoruzca2000}).   Maximal topologies are not
analytic.  In fact, analytic topologies cannot be  extremely
disconnected or irresolvable, nevertheless there are nodec regular
spaces with analytic topology \cite{MurgasUzca2019,Todoruzca2014}.

Nodec regular spaces are not easy to construct, the most common
examples are the maximal spaces. In \cite{MurgasUzca2019} were
constructed several examples of nodec regular spaces with analytic
topology that are neither SS nor $q^+$.   However, it was left
open whether there are nodec regular spaces satisfying other
boolean combinations of the properties $q^+$  and SS.  Using  MA
for countable posets, we will construct such nodec (in fact
maximal) regular spaces.  We  do not know if such spaces can be
found with analytic topology. We notice, that analogously to what
happens with $Q$-points, it is consistent that there are no
maximal $q^+$-spaces.

We actually show that (under MA for countable posets) for a given
hereditarily meager ideal $\mathcal{I}$ on $\omega$ there is a
maximal regular topology $\tau^\mathcal{I}$ on $\omega$ which is
crowded in a strong sense, namely, $\tau^\mathcal{I}\cap
\mathcal{I}=\{\emptyset\}$ and, moreover, if $\mathcal{I}$ is
$p^+$ then $\tau^\mathcal{I}$ is SS and if $\mathcal{I}$ is $q^+$,
so is $\tau^\mathcal{I}$.

\section{Preliminaries}

For us \emph{space} will always refer to a countable $T_1$
topological space.

$X$ will denote an infinite countable set. Note that we can use
characteristic functions to view any collection $\mathcal{A}$ of
subsets of $X$ as a subset of $2^X$. So we will say that such a
collection $\mathcal{A}$ is closed, $F_\sigma$, analytic, etc., if
$\mathcal{A}$ has the corresponding property when it is viewed as
a subspace of $2^X$ with the usual product topology. The
collection $\mathcal{A}$ is said to be \emph{hereditary} if
$B\subseteq A \in \mathcal{A}$ implies $B \in \mathcal{A}$.

An \emph{ideal} $\mathcal{I}$ on $X$ is a hereditary collection of
subsets of $X$ which is also closed under taking finite unions. We
denote by $\mathcal{I}^+$ the $\mathcal{I}$-positive subsets of
$X$ (i.e. $\mathcal{I}^+=\mathcal{P}(X) \setminus \mathcal{I}$).
Throughout this article, unless explicitly stated,  we will always
assume our ideals to be \emph{proper} (i.e. $X \notin
\mathcal{I}$) and \emph{free} (i.e. $[X]^{<\omega} \subseteq
\mathcal{I}$ where $[X]^{<\omega} $ denote the collection of
finite subsets of $X$).

We say that an ideal $\mathcal{I}$ is \emph{hereditarily meager}
(HM) if for any $A \in \mathcal{I}^+$ the ideal $\mathcal{I}\res
A=\mathcal{I} \cap \mathcal{P}(A)$ is meager as a subset of $2^A$.
Any analytic or co-analytic ideal (in particular any Borel ideal)
is hereditarily meager.

Given a finite $F \subseteq X$ we write $F \upn = \{A \subseteq X
: F \subseteq A \}$. Note that if $\mathcal{C}$ is a closed and
hereditary collection of subsets of $X$ and $A \subseteq X$ is not
in $\mathcal{C}$ then there is a finite $F \subseteq A$ such that
$F \upn \cap \mathcal{C} = \emptyset$.

Following \cite{Solecki} we define for an ideal $\mathcal{I}$ on
$X$ the set
\begin{align*}
C(\mathcal{I})=\{&\mathcal{C} \subseteq \mathcal{P}(X) :
\mathcal{C} \mbox{ is closed, hereditary and} \\ & \mbox{such that
} \forall A \in \mathcal{I} \, \exists F \in [X]^{<\omega} \, A
\setminus F \in \mathcal{C} \}.
\end{align*}

The following is essentially proved in \cite[Lemma 2.7]{Solecki}:

\begin{lemma}\label{solecki}
If $\mathcal{I}$ is a HM ideal on $X$ and $A \in \mathcal{I}^+$
then there is a $\mathcal{C} \in C(\mathcal{I})$ such that $A
\setminus F \notin \mathcal{C}$ for any $F \in [X]^{<\omega}$.
\end{lemma}

We write $A\subseteq^*B$ if $A\setminus B$ is finite. An ideal
$\mathcal{I}$ is $p^+$, if for every decreasing sequence $(A_n)_n$
of sets in $\mathcal{I}^+$, there is $A\in \mathcal{I}^+$ such
that $A\subseteq^* A_n$ for all $n\in\omega$. We say that
$\mathcal{I}$   is $p^-$, if for every decreasing sequence
$(A_n)_n$ of sets in $\mathcal{I}^+$ such that $A_n\setminus
A_{n+1}\in \mathcal{I}$, there is $B\in \mathcal{I}^+$ such that
$B\subseteq^* A_n$ for all $n$.  An ideal $\mathcal{I}$ is $q^+$,
if for every $A\in \mathcal{I}^+$ and every partition $(F_n)_n$ of
$A$ into finite sets, there is $S\in\mathcal{I}^+$ such that
$S\subseteq A$ and $S\cap F_n$  has at most one element for each
$n$.

For any topological space $X$ and $x\in X$, let $\mathcal{I}_x$ be
the ideal of all subsets $A$ of $X$ such that $x\not\in
\overline{A\setminus\{x\}}$. A point $x$ is $p^+$ if
$\mathcal{I}_x$ is $p^+$. The space $X$ is $p^+$ if every point is
$p^+$. We define analogously the notions of $p^-$ and $q^+$ points
(spaces).

A space  $X$ is {\em discretely generated} (DG) if for every
$A\subseteq X$ and $x\in\overline A$, there is $E\subseteq A$
discrete such that $x\in \overline E$. This notion was introduced
by Dow et al. in \cite{DTTW2002}. It is not easy to construct
spaces which are not DG, the typical examples are maximal spaces
(which are  nodec).

Every  countable $p^-$ regular space is  selectively separable
and discretely generated (see
\cite{CamargoUzca2018b,MurgasUzca2019}). In summary, we have the
following implications  for  countable regular spaces.

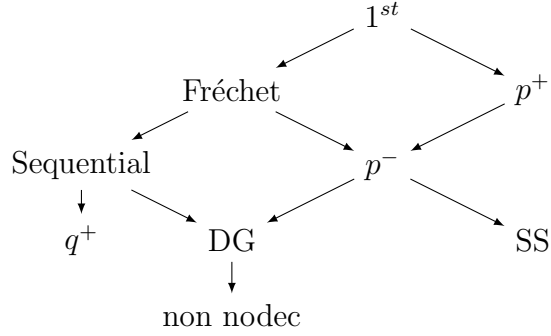
\begin{center}
\begin{tikzpicture}
  \node (1) at (4,4) {$1^{st}$};
  \node (f) at (2,3) {Fr\'echet};
  \node (p+) at (6,3) {$p^+$};
  \node (s) at (0,2) {Sequential};
  \node (p-) at (4,2) {$p^-$};
  \node (q) at (0,1) {$q^+$};
  \node (dg) at (2,1) {DG};
  \node (ss) at (6,1) {SS};
  \node (nod) at (2,0) {non nodec};
  \draw[-latex] (1) edge (f) (f) edge (s) (s) edge (q) (1) edge (p+) (p+) edge (p-) (p-) edge (ss) (p-) edge (dg) (dg) edge (nod) (f) edge (p-) (s) edge (dg);
\end{tikzpicture}
\end{center}

Thus, among all these properties,  a maximal space  can  only  be
$q^+$ or SS.

\bigskip

We denote by $\mathfrak{m}_c$ the minimal cardinal $\kappa$ for
which $\mathrm{MA}(\kappa)$ for countable posets fails.  The
statement that $\mathfrak{m}_c=\mathfrak{c}$ is also denoted
$\mathrm{MA_{ctble}}$. It is known that $ \mathfrak{m}_c\leq
\mathfrak{d}$. For a given space $X$, we write $w(X)$ and $\pi
w(X)$ for its weight and $\pi$-weight respectively. The following
facts are known. If $w(X)< \mathfrak{d}$, then $X$ is DG
(\cite{Murtinova2006}), SS  (\cite{Scheeper99} and see also
\cite{BarmanDow2011}) and resolvable (\cite{Cancinoetal}).

\bigskip

For the sake of completeness we give some known examples of
regular topologies to illustrate the  properties DG, $q^+$ and SS.

For each ideal $\mathcal{I}$ on $\omega$ consider the following topology  $\taui$ over $\seq$, the collection of finite sequences on $\omega$. A subset $U$ of $\seq$ is  open if and only if
$\{n\in \omega: \; s\widehat{\;\;}n \not\in U\}\in \mathcal{I}
\;\;\mbox{for all $s\in U$.}$
Let $Seq(\mathcal{I})$ denote the space $(\seq,\taui)$. This space  is
$T_2$, zero dimensional and crowded. Notice that
when $\mathcal{I}$ is the ideal of finite sets, then  $Seq(\mathcal{I})$ is homeomorphic to the Arkhanglel'ski{\~{\i}}-Franklin space $Seq$.  When  $\mathcal{I}$ is analytic, so is $\taui$.  We also know that  $Seq(\mathcal{I})$ is DG and  non SS for any $\mathcal{I}$. It  is $q^+$ if, and only if, $\mathcal{I}$ is $q^+$ (\cite{CamargoUzca2018b}).

Next we show a space which is SS, DG and not $q^+$ that will be used later in the paper.

\begin{example}\label{clopen}
Let  $CL(2^\omega)$ be the collection of all clopen subsets of
$2^\omega$  as a subspace of $2^{2^\omega}$. Then $CL(2^\omega)$
is $p^+$ (and thus SS and DG) and not $q^+$ (see
\cite{CamargoUzca2018b}). We will need later a subspace of
$CL(2^\omega)$ where $q^+$ fails in a somewhat stronger form.  Let
$2^{<\omega}$ be the collection of finite binary sequences. If
$s\in 2^{<\omega}$ and $i\in \{0,1\}$, $|s|$ denotes its length
and $ s\widehat{\;\;}i$ the sequence obtained concatenating $s$
with $i$.   For $s\in 2^{<\omega}$ and $\alpha\in 2^\omega$, let
$s\prec \alpha$ if $\alpha(i)=s(i)$ for all $i<|s|$ and
$[s]=\{\alpha\in 2^\omega: \; s\prec \alpha\}$. Each $x\in
CL(2^\omega)$ is a finite union of sets of the form $[s]$ for
$s\in 2^{<\omega}$. Let $s,t\in 2^n$, we say that $s$ and $t$ are
{\em linked} if there is a sequence $u\in 2^{n-1}$ such that
$s=u\widehat{\;\;}i$ and $t=u\widehat{\;\;}j$ with $i+j=1$.  For
$k$ a positive integer,  we say that  a $x\in X$  is {\em
$k$-adequated}, if   $x$ can be written as $[s_1]\cup\cdots\cup
[s_m]$  with each $s_i\in 2^k$  and  such that any pair of them
are not linked. Let $A_k=\{x\in CL(2^{\omega}): \; x \;\mbox{is
$k$-adequated}\}$.  Notice that $\bigcup_{k\in\omega}A_{k+1}$ is
dense in $CL(2^\omega)$. Let $X=\bigcup_{k\in \omega }A_{k+1}$.
Let $S\subseteq X$ be such that $S\cap A_k$ has at most one
element for each $k$. Then $S$ is closed in $X$. This clearly
shows that $X$ is not $q^+$.
\end{example}

As we have already mentioned, analytic nodec spaces are hard to
define and they are the only examples we know of non DG analytic
spaces. In  \cite{MurgasUzca2019} was constructed an analytic
regular space $\mathbb{Y}(\mathcal{I}_{nd}^{**})$ which is nodec,
non $q^+$ and non SS. However, we do not know if there is an
analytic nodec $q^+$ (or SS) regular space.

In the following table we summarize what we know. One of the main
goals of this paper is to construct the maximal spaces  mentioned
in the last column. Notice that the existence of those maximal
spaces (except the one at the bottom row) is not provable in ZFC.

\bigskip

\begin{tabular}{|c|c|c|c|c|}
\hline
DG & SS & $q^+$ & Analytic  topology & Non definable topology\\
\hline \hline
$\checkmark$ & $\checkmark$ & $\checkmark$  & $\mathbb{Q}$  & \\ \hline
\checkmark & \checkmark & $\times$ & $Cl(2^{\mathbb{N}})$ &  \\ \hline
\checkmark & $\times$ & \checkmark &  $Seq$& \\ \hline
\checkmark & $\times$ & $\times$  &  $Seq(\mathcal{I})$ (with $\mathcal{I}$ non $q^+$)&  \\ \hline
$\times$ & \checkmark & \checkmark  &  {??}  & ($\mathrm{MA_{ctble}}$) Maximal \\ \hline
$\times$ & \checkmark & $\times$ & {??} & ($\mathrm{MA_{ctble}}$) Maximal  \\ \hline
$\times$ & $\times$ & \checkmark  &  {??} & ($\mathrm{MA_{ctble}}$) Maximal  \\ \hline
$\times$ & $\times$ & $\times$ & $\mathbb{Y}(\mathcal{I}_{nd}^{**})$ &  Maximal  \\ \hline
\end{tabular}

\section{$\mathcal{I}$-crowded topologies}

When constructing a crowded topology on a set $X$ we only need to
worry not to include any finite set in the topology, that is, to
keep the open sets outside of the ideal $[X]^{<\omega}$. In order
to get additional properties on the topology it will be useful to
follow this idea with other ideals.

\begin{definition} Given a space $(X,\tau)$ and an ideal
$\mathcal{I}$ on $X$ we say that:
\begin{itemize}
\item The topology $\tau$ is \emph{$\mathcal{I}$-crowded} if $\tau \cap
\mathcal{I}=\{\emptyset\}$.

\item A subset $A \subseteq X$ is \emph{$(\mathcal{I},\tau)$-crowded} if
for every $U \in \tau$ the intersection $A \cap U$ is either empty
or belongs to $\mathcal{I}^+$.
\end{itemize}
\end{definition}

Note that saying that $\tau$ is $\mathcal{I}$-crowded is
equivalent to say that $X$ is $(\mathcal{I},\tau)$-crowded. For
future reference we collect a few simple observations in the
following remark which the reader can easily corroborate.

\begin{remark}\label{remark on ideals} \
\begin{enumerate}
\item A topology $\tau$ is crowded if and only if it is
$[X]^{<\omega}$-crowded if and only if it is $\mathcal{I}$-crowded
for some $\mathcal{I}$.

\item If $\tau$ has a $\pi$-net (in particular if it has a base or
a $\pi$-base) contained in $\mathcal{I}^+$ then $\tau$ is
$\mathcal{I}$-crowded.

\item The union of any $\subseteq$-chain of $\mathcal{I}$-crowded
topologies on $X$ generates an $\mathcal{I}$-crowded topology on
$X$.

\item If $A$ is $(\mathcal{I},\tau)$-crowded and $A \subseteq B
\subseteq cl_\tau(A)$ then $B$ is also
$(\mathcal{I},\tau)$-crowded.
\end{enumerate}
\end{remark}

The following lemma allows us to add certain subsets of $X$ to a
given $\mathcal{I}$-crowded topology on $X$ while keeping the
topology $\mathcal{I}$-crowded.

\begin{lemma}\label{partition implies extension ideals}
Let $(X,\tau)$ be a zero-dimensional $\mathcal{I}$-crowded space
and $A \subseteq X$. Suppose that $A \subseteq X$ admits a
partition $A=\bigcup_{m \in \omega} A_m$ such that
$\overbar{A}_m=\overbar{A}$ and $A_m$ is
$(\mathcal{I},\tau)$-crowded for all $m \in \omega$. Then there is
a zero-dimensional $\mathcal{I}$-crowded topology $\tau'\supseteq
\tau$ on $X$ such that $w(\tau')\leq w(\tau)$ and $A \in \tau'$.
Moreover $cl_{\tau'}(A)=cl_\tau(A)$.
\end{lemma}

\begin{proof}
Fix a base $\mathcal{B}$ of clopen subsets for $\tau$, closed
under finite intersections and with $|\mathcal{B}|=w(\tau)$.
Define
$$\mathcal{B}'=\mathcal{B} \cup \left\{A_m : m \in \omega \right\}
\cup \left\{ X \setminus A_m : m \in \omega \right\}.$$ Since
$\mathcal{B} \subseteq \mathcal{B}'$ and $|\mathcal{B}'|
=|\mathcal{B}|$ we have that $\mathcal{B}'$ is a subbase for a
topology $\tau' \supseteq \tau$ with $w(\tau')\leq w(\tau)$.
Clearly all the elements of $\mathcal{B}'$ are $\tau'$-clopen and
therefore $\tau'$ is zero-dimensional. Note that $A \in \tau'$
being a union of elements of $\mathcal{B}'$. To show that $\tau'$
is $\mathcal{I}$-crowded, note that if $V$ is a non-empty finite
intersection of elements of $\mathcal{B}'$ then there is $U \in
\mathcal{B}$ such that $V$ is of the form $V=U \cap A_m$ for some
$m \in \omega$ or $V=U \cap \bigcap \{(X \setminus A_j): j \in
J\}$ for some finite $J \subseteq \omega$. If $U \cap A=
\emptyset$ then $V=U$ which is in $\mathcal{I}^+$ since $\tau$ is
$\mathcal{I}$-crowded. If $U \cap A \neq \emptyset$ then by
assumption $U \cap A_i$ is non-empty and belongs to
$\mathcal{I}^+$ for every $i \in \omega$ and therefore $V \in
\mathcal{I}^+$ since it contains at least one of these. To see that $cl_{\tau'}(A)=cl_\tau(A)$ note that if $x \in cl_\tau(A)$ and $V$ is a basic $\tau'$-neighborhood of $x$, the previous argument shows in particular that $V \cap A \neq \emptyset$ and hence $x \in cl_{\tau'}(A)$.
\end{proof}

Perhaps it is well known that a crowded space of weight smaller
than $\mathfrak{m}_c$ is $\omega$-resolvable (i.e. it can be
partitioned into countably many disjoint dense subsets). Next we
prove a stronger result.

\begin{lemma}\label{resolvability ideals}
Let $\mathcal{I}$ be a HM ideal on $X$ and $\tau$ an $\mathcal{I}$-crowded topology on $X$ with
$w(\tau)<\mathfrak{m}_c$. Suppose that $A \subseteq X$ is
$(\mathcal{I},\tau)$-crowded. Then there is a partition
$A=\bigcup_{m \in \omega} A_m$ such that each
$A_m$ is $(\mathcal{I},\tau)$-crowded and dense in $A$.
\end{lemma}

\begin{proof}
Let $\mathbb{P}$ be the set of all finite partial functions $p: A
\to \omega$ ordered by reverse inclusion and fix $\mathcal{B}$ a
base for $\tau$ with $|\mathcal{B}|<\mathfrak{m}_c$. For every $U
\in \mathcal{B}$ for which $A \cap U$ is non-empty we use Lemma
\ref{solecki} to fix $\mathcal{C}_U \in C(\mathcal{I})$ such that
$(A \cap U) \setminus F \notin \mathcal{C}_U$ for any finite $F
\subseteq X$. Now for every such $U$, every $F \in [X]^{<\omega}$
and every $m \in \omega$, the set
$$
\mathbb{D}(U,F,m)=\{p \in \mathbb{P}
: (p^{-1}(m) \cap U \setminus F)\upn \cap \mathcal{C}_U= \emptyset\}
$$
is dense in $\mathbb{P}$.

To see this, fix $U,F,m$ as before and $q \in \mathbb{P}$. Since
$A \cap U \setminus (F \cup \dom q) \notin \mathcal{C}_U$, using
that $\mathcal{C}_U$ is closed and hereditary, we can find a
finite $E \subseteq A \cap U \setminus (F \cup \dom q)$ such that
$E\upn \cap \mathcal{C}_U = \emptyset$. But now $p=q \cup (E
\times \{m\})$ is an extension of $q$ that belongs to
$\mathbb{D}(U,F,m)$.

Since $\mathbb{P}$ is countable and
$|\mathcal{B}|<\mathfrak{m}_c$, there exists a filter $G \subseteq
\mathbb{P}$ intersecting all the $\mathbb{D}(U,F,m)$ and also
intersecting the dense sets $\{p \in \mathbb{P} : a \in \dom p \}$
for $a \in A$. Thus $\bigcup G :A \to \omega$ is a total function
and $A_m=(\bigcup G)^{-1}(m)$ for $m \in \omega$ defines a
partition of $A$ as desired.

To see that $A_m$ is $(\mathcal{I},\tau)$-crowded and dense in $A$
let $U \in \mathcal{B}$ such that $A \cap U \neq \emptyset$. We
want to show that $A_m \cap U \in \mathcal{I}^+$, so fix a finite
$F \subseteq X$ and choose a $p \in G \cap \mathbb{D}(U,F,m)$.
Then $(p^{-1}(m) \cap U \setminus F)\upn \cap \mathcal{C}_U =
\emptyset$ and therefore $A_m \cap U \setminus F \notin
\mathcal{C}_U$. Since this is true for any finite $F$ and
$\mathcal{C}_U$ belongs to $C(\mathcal{I})$, it follows that $A_m
\cap U \notin \mathcal{I}$.
\end{proof}

Note that in particular this last lemma tells us that $X$ can be
partitioned into countably many dense subsets, so $(X,\tau)$ is
$\omega$-resolvable in a strong sense: the dense subsets can be
chosen outside of the ideal $\mathcal{I}$. For the ideal
$\mathcal{I}=[X]^{<\omega}$ we just get:

\begin{corollary}
Any crowded space of weight smaller than $\mathfrak{m}_c$ is $\omega$-resolvable.
\end{corollary}

We can also improve Lemma \ref{partition implies extension ideals}
for spaces of small weight. The next corollary will be key in our
main construction.

\begin{corollary}\label{extension ideals}
Let $\mathcal{I}$ be a HM ideal on $X$ and $\tau$ a
zero-dimensional $\mathcal{I}$-crowded topology on $X$ with
$w(\tau)<\mathfrak{m}_c$. Suppose that $A \subseteq X$ is
$(\mathcal{I}, \tau)$-crowded. Then there is a zero-dimensional
$\mathcal{I}$-crowded topology $\tau'\supseteq \tau$ on $X$ such
that $w(\tau')\leq w(\tau)$ and $A \in \tau'$.   Moreover
$cl_{\tau'}(A)=cl_\tau(A)$.
\end{corollary}

\begin{proof}
It follows immediately from Lemmas \ref{partition implies
extension ideals} and \ref{resolvability ideals}.
\end{proof}

The last two observations on Remark \ref{remark on ideals} and a
repeated application of the previous result allows us to show:

\begin{lemma}\label{discrete ideals}
Let $\mathcal{I}$ be an HM ideal on $X$ and $\tau$ a
zero-dimensional $\mathcal{I}$-crowded topology on $X$ with
$w(\tau)<\mathfrak{m}_c$. Suppose that $D$ is an
$(\mathcal{I},\tau)$-crowded subset of $X$. Then there is a
zero-dimensional $\mathcal{I}$-crowded topology $\tau'\supseteq
\tau$ on $X$ such that $w(\tau')\leq w(\tau)$, $cl_\tau(D) \in
\tau'$ and $cl_\tau(D) \setminus D$ is $\tau'$-discrete.
\end{lemma}

\begin{proof}
Suppose that $cl_\tau(D) \setminus D$ is infinite (otherwise we
can just let $\tau'$ be the topology given by Corollary
\ref{extension ideals} applied to $\tau$ and $cl_\tau(D)$) and fix
an enumeration $cl_\tau(D) \setminus D=\{a_n:n \in \omega\}$. Let
$\tau_0=\tau$ and given $\tau_n$ let $\tau_{n+1}$ be the topology
given by Corollary \ref{extension ideals} applied to the space
$(X,\tau_n)$ and the subset $D \cup \{a_i:i \leq n\}$. It is clear
that all the $\tau_n$'s are $\mathcal{I}$-crowded,zero-dimensional
and $w(\tau_n) \leq w(\tau)$. So if we let $\tau'$ be the topology
generated by $\bigcup_{n \in \omega} \tau_n$ we also get a
zero-dimensional $\mathcal{I}$-crowded topology with $w(\tau')
\leq w(\tau)$. Moreover, since the set $D \cup \{a_i:i \leq n\}$
is $\tau'$-open for each $n \in \omega$, $cl_\tau(D) \setminus D$
is $\tau'$-discrete and $cl_\tau(D)$ is $\tau'$-open.
\end{proof}

Note that the conclusion of the previous result guarantees that
$cl_{\tilde{\tau}}(D)=cl_\tau(D)$ for any crowded topology
$\tilde{\tau} \supseteq \tau'$.

The following two lemmas generalize the fact that any space of
weight smaller than $\mathfrak{m}_c$ is both selectively separable
and $q^+$.

\begin{lemma}\label{SS ideals}
Let $\mathcal{I}$ be a $p^+$ and HM ideal on $X$ and $\tau$ an
$\mathcal{I}$-crowded topology on $X$ with $\pi w(\tau)<
\mathfrak{m}_c$. Suppose that $\langle D_i :i \in \omega \rangle$
is a decreasing sequence of $(\mathcal{I},\tau)$-crowded dense
subsets of $X$. Then there are finite sets $K_i \subseteq D_i$ for
$i\in \omega$ such that $\bigcup_{i \in \omega}K_i$ is dense and
$(\mathcal{I},\tau)$-crowded.
\end{lemma}

\begin{proof}
Let $D=\bigcup_{i \in \omega}D_i$ and let $\mathbb{P}$ be the set
of all finite partial functions $p:\omega \to [D]^{<\omega}$ such
that $p(i) \subseteq D_i$ for all $i \in \dom p$. We order
$\mathbb{P}$ by reverse inclusion.

Fix $\mathcal{B}$ a $\pi$-base for $\tau$ with
$|\mathcal{B}|<\mathfrak{m}_c$. Since each $D_i$ is
$(\mathcal{I},\tau)$-crowded and dense, using that $\mathcal{I}$
is a $p^+$-ideal, we can find for each $U \in \mathcal{B}$ a set
$D_U \in \mathcal{I}^+$ such that $D_U \subseteq^* D_i \cap U$ for
all $i \in \omega$. Using Lemma \ref{solecki} we can fix
$\mathcal{C}_U \in C(\mathcal{I})$ such that $D_U \setminus F
\notin \mathcal{C}_U$ for any finite $F \subseteq X$.

For each $U \in \mathcal{B}$ and each finite $F \subseteq X$, the
set $$\mathbb{D}(U,F)=\{p \in \mathbb{P} : (\exists i \in \dom p)
\, (p(i) \cap U \setminus F)\upn \cap \mathcal{C}_U=\emptyset \}$$
is dense in $\mathbb{P}$. To see this, fix $U \in \mathcal{B}$, $F
\in [X]^{<\omega}$ and $q \in \mathbb{P}$. Choose $i \in \omega
\setminus \dom q$. Since $D_U \setminus (D_i \cap U)$ is finite we
have that $D_U \cap D_i \cap U \setminus F \notin \mathcal{C}_U$.
Using that $\mathcal{C}_U$ is closed and hereditary we can find a
finite $E \subseteq D_i \cap U \setminus F$ such that $E \upn \cap
\mathcal{C}_U = \emptyset$. But now $p=q \cup \{(i,E)\}$ is an
extension of $q$ that belongs to $\mathbb{D}(U,F)$.

Since $\mathbb{P}$ is countable and
$|\mathcal{B}|<\mathfrak{m}_c$, there exists a filter $G \subseteq
\mathbb{P}$ intersecting all the $\mathbb{D}(U,F)$ and also
intersecting the dense sets $\{p \in \mathbb{P} : i \in \dom p \}$
for $i \in \omega$. Thus $\bigcup G :\omega \to [D]^{<\omega}$ is
a total function and the sets $K_i=(\bigcup G)(i)$ for $i \in
\omega$ satisfy the conclusion of the lemma, since one can see
that $U \cap \bigcup_{i \in \omega}K_i \setminus F \notin
\mathcal{C}_U$ for all $F \in [X]^{<\omega}$ and all $U \in
\mathcal{B}$.
\end{proof}

Using the previous result with the ideal
$\mathcal{I}=[X]^{<\omega}$ we get:

\begin{corollary}
Any crowded space of $\pi$-weight smaller than $\mathfrak{m}_c$ is
selectively separable.
\end{corollary}

Now we show an analogous result concerning the property $q^+$.

\begin{lemma}\label{q ideals}
Let $\mathcal{I}$ be a $q^+$ and HM ideal on $X$ and $\tau$ an
$\mathcal{I}$-crowded topology on $X$ with $w(\tau)<
\mathfrak{m}_c$. If $\langle F_i : i \in \omega \rangle$ is a
sequence of pairwise disjoint finite subsets of $X$ such that
$\bigcup_{i \in \omega}F_i$ is $(\mathcal{I},\tau)$-crowded, then
there is an $(\mathcal{I},\tau)$-crowded $S \subseteq \bigcup_{i
\in \omega}F_i$ such that $|S \cap F_i| \leq 1$ for each $i \in
\omega$ and $\overbar{S}=\overbar{\bigcup_{i \in \omega}F_i}$.
\end{lemma}

\begin{proof}
Just for the purpose of this proof let us say that a subset of $X$
is a \emph{selector} if it intersects each $F_i$ in at most one
point. Let $\mathbb{P}$ be the set of all finite selectors $p
\subseteq \bigcup_{i \in \omega}F_i$ ordered by reverse inclusion.
Also fix $\mathcal{B}$ a base for $\tau$ with
$|\mathcal{B}|<\mathfrak{m}_c$.

For each $U \in \mathcal{B}$ which intersects $\bigcup_{i \in
\omega}F_i$ we have that $(\bigcup_{i \in \omega}F_i) \cap U \in
\mathcal{I}^+$ and since $\mathcal{I}$ is a $q^+$-ideal there is a
selector $S_U \subseteq (\bigcup_{i \in \omega}F_i) \cap U$ with
$S_U \in \mathcal{I}^+$. Using Lemma \ref{solecki} we can fix
$\mathcal{C}_U \in C(\mathcal{I})$ such that $S_U \setminus F
\notin \mathcal{C}_U$ for any finite $F \subseteq X$. Now for each
$F \in [X]^{<\omega}$ the set $$\mathbb{D}(U,F)=\{p \in \mathbb{P}
: (p \cap U \setminus F)\upn \cap \mathcal{C}_U = \emptyset\}$$ is
dense in $\mathbb{P}$.

To see this, fix $U,F$ as before and $q \in \mathbb{P}$. Since
$S_U \setminus F \setminus q  \notin \mathcal{C}_U$ and
$\mathcal{C}_U$ is closed and hereditary, we can find a finite $s
\subseteq S_U \setminus F \setminus q$ such that $s\upn \cap
\mathcal{C}_U = \emptyset$. But now $p=q \cup s$ is an extension
of $q$ that belongs to $\mathbb{D}(U,F)$.

Since $\mathbb{P}$ is countable and
$|\mathcal{B}|<\mathfrak{m}_c$, there exists a filter $G \subseteq
\mathbb{P}$ intersecting all the $\mathbb{D}(U,F)$ and we can let
$S=\bigcup G$. Clearly $S$ is a selector.

To see that $S$ is $(\mathcal{I},\tau)$-crowded and
$\overbar{S}=\overbar{\bigcup_{i \in \omega}F_i}$, fix $U \in
\mathcal{B}$ which intersects $\bigcup_{i \in \omega}F_i$ and a
finite $F \subseteq X$. Then there is a $p \in G$ such that $(p
\cap U \setminus F)\upn \cap \mathcal{C}_U = \emptyset$ and
therefore $S \cap U \setminus F \notin \mathcal{C}_U$. Since this
is true for any finite $F$ and $\mathcal{C}_U$ belongs to
$C(\mathcal{I})$, it follows that $S \cap U \notin \mathcal{I}$.
\end{proof}

Using $\mathcal{I}=[X]^{<\omega}$ again, we obtain:

\begin{corollary}
Any space is $q^+$ at each point of character smaller than
$\mathfrak{m}_c$.
\end{corollary}

\begin{proof}
Let $Y$ be a space, $y \in Y$, $\mathcal{B}$ a local base at $y$
with $|\mathcal{B}|< \mathfrak{m}_c$ and suppose that $y \in
\overline{\bigcup_{n \in \omega} F_n} \setminus \bigcup_{n \in
\omega} F_n $ where each $F_n$ is a finite subset of $Y$. Use the
previous lemma with the space $X=\{y\} \cup \bigcup_{n \in \omega}
F_n$ with the topology generated by $\{U \cap Y : U \in
\mathcal{B}\} \cup \{\bigcup_{n \in \omega} F_n\}$ and the ideal
$[X]^{<\omega}$.
\end{proof}

\section{Maximal $\mathcal{I}$-crowded topologies}

Depending on the ideal $\mathcal{I}$, there might or might not
exist a nice $\mathcal{I}$-crowded topology on $X$. If
$\mathcal{I}=\{\emptyset\}$ then any topology on $X$ is
$\mathcal{I}$-crowded, but if $\mathcal{I}=\mathcal{P}(X)$ then
only the trivial topology is $\mathcal{I}$-crowded. For
$\mathcal{I}=[X]^{<\omega}$, a topology on $X$ is
$\mathcal{I}$-crowded if and only if it is crowded. If
$\mathcal{I}$ is a prime ideal then $\mathcal{I}^+ \cup
\{\emptyset\}$ is the finest $\mathcal{I}$-crowded topology on $X$
and it is not Hausdorff. Finally, for meager ideals we have the
following:

\begin{theorem}\label{homeomorphism}
Let $\mathcal{I}$ be a meager ideal on $X$ and let $(Y,\rho)$ be a
crowded space with $\pi w(\rho)<\mathfrak{m}_c$. Then there exists
an $\mathcal{I}$-crowded topology $\tau$ on $X$ such that
$(X,\tau)$ is homeomorphic to $(Y,\rho)$.
\end{theorem}

\begin{proof}
By the Jalali-Naini--Talagrand Theorem (see, for instance,
\cite[p. 33]{todortopics}), we can fix a partition $X=\bigcup_{n
\in \omega}F_n$ into finite subsets such that no element of
$\mathcal{I}$ contains infinitely many of the $F_n$'s. Lets also
fix a $\pi$-base $\mathcal{B}$ for $\rho$ with $|\mathcal{B}|<
\mathfrak{m}_c$.

Let $\mathbb{P}$ be the set of all finite partial injective
functions $p:X \to Y$ ordered by reverse inclusion. For each $U
\in \mathcal{B}$ and $k \in \omega$ the set $$\mathbb{D}(U,k)=\{p
\in \mathbb{P}: F_n \subseteq p^{-1}(U) \mbox{ for some }n \geq
k\}$$ is dense in $\mathbb{P}$. To see this fix $U \in
\mathcal{B}$, $k \in \omega$ and $q \in \mathbb{P}$. Find $n \geq
k$ such that $F_n \cap \dom q =\emptyset$ and choose any $E
\subseteq U \setminus \rng q$ with $|E|=|F_n|$ and a bijection
$r:F_n \to E$. Now $p=q \cup r$ is an extension of $q$ that
belongs to $\mathbb{D}(U,k)$.

Since $\mathbb{P}$ is countable and $|\mathcal{B}| <
\mathfrak{m}_n$, there is a filter $G \subseteq \mathbb{P}$
intersecting all the $\mathbb{D}(U,k)$ and also intersecting the
dense sets $\{p \in \mathbb{P}: x \in \dom p\}$ for $x \in X$ and
$\{p \in \mathbb{P}: y \in \rng p \}$ for $y \in Y$. Then
$f=\bigcup G : X \to Y$ is a bijection and for each $U \in
\mathcal{B}$ we have $f^{-1}(U) \in \mathcal{I}^+$. Therefore if
we let $\tau= \{f^{-1}(V) : V \in \rho \}$ we get an
$\mathcal{I}$-crowded topology on $X$ and $f$ is a homeomorphism
between $(X,\tau)$ and $(Y, \rho)$.
\end{proof}

\begin{question}
Can we remove the hypothesis on the $\pi$-weight of $\rho$?
\end{question}

We need one more lemma before we state and prove our main result.

\begin{lemma}\label{nocermax ideals}
Let $X$ be a maximal space, $A \subseteq X$ and $x \in X$. If $x
\in \overline{ A }\setminus A$, then there is  $B \subseteq A$
without isolated points such that  $x \in \overline{B}$.
\end{lemma}

\begin{proof}
Let $D=\{y \in A: y \text{ is an isolated point of }A \}$. Since
$D$ is discrete and $X$ is maximal, then $D$ is closed. Let
$B=A\setminus D$. We have $x \in \overline{ A } = \overline{
(A\setminus D) \cup D }=\overline{ B} \cup D$. Thus  $x \in
\overline{B}$. It is easy to see that $B$  has no isolated points.
\end{proof}

Now we are ready for the main result of this paper.

\begin{theorem}[$\mathfrak{m}_c=\mathfrak{c}$]\label{main}
For any HM ideal $\mathcal{I}$ on $X$ there exists a maximal,
zero-dimensional and $\mathcal{I}$-crowded topology
$\tau^\mathcal{I}$ on $X$ such that:
\begin{enumerate}
    \item[(i)] If $\mathcal{I}$ is $p^+$ then $\tau^\mathcal{I}$ is selectively separable.
    \item[(ii)] If $\mathcal{I}$ is $q^+$ then $\tau^\mathcal{I}$ is $q^+$.
\end{enumerate}
\end{theorem}

\begin{proof}
By Theorem \ref{homeomorphism} there is an $\mathcal{I}$-crowded
topology $\tau_0$ on $X$ such that $(X,\tau_0)$ is homeomorphic to
the rational numbers with their usual topology.

Now we fix an enumeration $\{\langle A^\alpha_i : i \in \omega
\rangle \}_{\alpha \in \mathfrak{c}}$ of all sequences of subsets
of $X$ and construct inductively an increasing sequence $\langle
\tau_\alpha : \alpha \leq \mathfrak{c}\rangle$ of topologies on
$X$ with the following properties:

\begin{enumerate}
\item Each $\tau_\alpha$ is zero-dimensional,
$\mathcal{I}$-crowded and $w(\tau_\alpha)\leq |\alpha|+\aleph_0 <
\mathfrak{c}$.

\item If $A^\alpha_0=\emptyset$ and $A^\alpha_1$ is
$(\mathcal{I},\tau_\alpha)$-crowded then $A^\alpha_1 \in
\tau_{\alpha+1}$.

\item If $A^\alpha_0=\emptyset$ and $A^\alpha_1$ is infinite but
not $(\mathcal{I},\tau_\alpha)$-crowded then $A^\alpha_1$ has
$\tau_{\alpha+1}$-isolated points.

\item If $A^\alpha_i$ is $\tau_\alpha$-dense for all $i \in
\omega$ and there is a sequence $\langle K_i\rangle_i \in
\prod_i[A^\alpha_i]^{<\omega}$ such that $\bigcup_{i \in
\omega}K_i$ is $(\mathcal{I},\tau)$-crowded and
$\tau_\alpha$-dense, then $X \setminus \bigcup_{i \in \omega}K_i$
is $\tau_{\alpha+1}$-discrete for some such $\langle
K_i\rangle_i$.

\item If $\langle A^\alpha_i\rangle_i$ is a sequence of pairwise
disjoint finite sets and there is an $S \subseteq \bigcup_i
A^\alpha_i$ which is $(\mathcal{I},\tau_\alpha)$-crowded,
$\tau_\alpha$-dense in $\bigcup_i A^\alpha_i$ and $|S \cap
A^\alpha_i|\leq 1$ for all $i\in \omega$, then
$cl_{\tau_\alpha}(S) \in \tau_{\alpha+1}$ and $cl_{\tau_\alpha}(S)
\setminus S$ is $\tau_{\alpha+1}$-discrete for some such $S$.
\end{enumerate}

Observe that at a limit ordinal $\lambda \leq \mathfrak{c}$ we can
just let $\tau_\lambda$ be the topology generated by
$\bigcup_{\alpha \in \lambda}\tau_\alpha$. Only condition 1 needs
to be checked but this is easy. What we do at a successor ordinal
$\alpha+1$ depends on which of the hypothesis on conditions 2, 3,
4 and 5 is satisfied by $\langle A^\alpha_i : i \in \omega
\rangle$ (note that these conditions are mutually exclusive and if
neither is satisfied we just let $\tau_{\alpha+1}=\tau_\alpha$):

\begin{itemize}
\item If $A^\alpha_0= \emptyset$ and $A^\alpha_1$ is
$(\mathcal{I},\tau_\alpha)$-crowded we let $\tau_{\alpha+1}$ be
the topology given by Lemma \ref{extension ideals} for
$\tau_\alpha$ and $A^\alpha_1$.

\item If $A^\alpha_0=\emptyset$ and $A^\alpha_1$ is infinite but
not $(\mathcal{I},\tau_\alpha)$-crowded then there exists $U \in
\tau_\alpha$ such that $A^\alpha_1 \cap U \in \mathcal{I}
\setminus \{\emptyset\}$. Fix $x \in A^\alpha_1 \cap U$ and let $A
= X \setminus (A^\alpha_1 \cap U) \cup \{x\}$. Note that for any
$V \in \tau_\alpha$ we have $A \cap V \supseteq V \setminus
(A^\alpha_1 \cap U) \in \mathcal{I}^+$ (since $\tau_\alpha$ is
$\mathcal{I}$-crowded) and hence $A$ is
$(\mathcal{I},\tau_\alpha)$-crowded. Thus we can let
$\tau_{\alpha+1}$ be the topology given by Corollary
\ref{extension ideals} for $\tau_\alpha$ and $A$. Now since $A$
and $U$ are both $\tau_{\alpha+1}$-open, we get that $x$ is
$\tau_{\alpha+1}$-isolated in $A^\alpha_1$.

\item If $A^\alpha_i$ is $\tau_\alpha$-dense for all $i \in
\omega$ and there is a sequence $\langle K_i\rangle_i \in
\prod_i[A^\alpha_i]^{<\omega}$ such that $\bigcup_{i \in
\omega}K_i$ is $(\mathcal{I},\tau)$-crowded and
$\tau_\alpha$-dense, we can let $\tau_{\alpha+1}$ be the topology
given by Lemma \ref{discrete ideals} for $\tau_\alpha$ and
$\bigcup_{i \in \omega}K_i$. Since $\bigcup_{i \in \omega}K_i$ is
$\tau_\alpha$-dense, we get that $X \setminus \bigcup_{i \in
\omega}K_i$ is $\tau_{\alpha+1}$-discrete.

\item If $\langle A^\alpha_i\rangle_i$ is a sequence of pairwise
disjoint finite sets and there is an $S \subseteq \bigcup_i
A^\alpha_i$ which is $(\mathcal{I},\tau_\alpha)$-crowded,
$\tau_\alpha$-dense in $\bigcup_i A^\alpha_i$ and $|S \cap
A^\alpha_i|\leq 1$ for all $i\in \omega$, then we can let
$\tau_{\alpha+1}$ be the topology given by Lemma \ref{discrete
ideals} for $\tau_\alpha$ and $S$. Then $cl_{\tau_\alpha}(S) \in
\tau_{\alpha+1}$ and $cl_{\tau_\alpha}(S) \setminus S$ is
$\tau_{\alpha+1}$-discrete.
\end{itemize}

Now we show that $\tau^\mathcal{I}=\tau_\mathfrak{c}$ has all the
properties that we want. It is immediate from the construction
that $\tau^\mathcal{I}$ is zero-dimensional and
$\mathcal{I}$-crowded.

If $A \subseteq X$ has no $\tau^\mathcal{I}$-isolated points, find
$\alpha \in \mathfrak{c}$ such that $A^\alpha_0=\emptyset$ and
$A^\alpha_1=A$. Since $\tau_{\alpha+1} \subseteq
\tau^\mathcal{I}$, $A$ has no $\tau_{\alpha+1}$-isolated points
and by condition 3 we conclude that $A$ is
$(\mathcal{I},\tau_\alpha)$-crowded and therefore, by condition 2,
$A \in \tau^\mathcal{I}$. This shows that $\tau^\mathcal{I}$ is
maximal (note that $\tau^\mathcal{I}$ is crowded by the first
observation on Remark \ref{remark on ideals}).

To prove (i) suppose that $\mathcal{I}$ is $p^+$ and fix a
decreasing sequence $\langle D_i : i \in \omega \rangle$ of
$\tau^\mathcal{I}$-dense subsets of $X$. Since dense subsets of a
maximal space are necessarily open we have that each $D_i$ is
$\tau^\mathcal{I}$-open and therefore
$(\mathcal{I},\tau^\mathcal{I})$-crowded. Thus if we take $\alpha
\in \mathfrak{c}$ such that $\langle A^\alpha_i : i \in \omega
\rangle = \langle D_i : i \in \omega \rangle$, we have that each
$A^\alpha_i$ is $(\mathcal{I},\tau_\alpha)$-crowded and
$\tau_\alpha$-dense. Thus Lemma \ref{SS ideals} tells us that
there exist finite sets $K_i \subseteq A^\alpha_i$ for $i \in
\omega$ such that $\bigcup_{i \in \omega}K_i$ is
$\tau_\alpha$-dense and $(\mathcal{I},\tau_\alpha)$-crowded. But
now condition 4 guarantees that $X \setminus \bigcup_{i \in
\omega}K_i$ is $\tau^\mathcal{I}$-discrete for some such $\langle
K_i \rangle_i$ and hence $\bigcup_{i \in \omega}K_i$ is dense,
showing that $\tau^\mathcal{I}$ is selectively separable.

Finally we prove (ii). Suppose that $\mathcal{I}$ is $q^+$, let
$\langle K_i : i \in \omega \rangle$ be a sequence of pairwise
disjoint finite subsets of $X$ and suppose that $x \in
\overbar{\bigcup_{i \in \omega}K_i} \setminus \bigcup_{i \in
\omega}K_i$ (here closures are taken with respect to
$\tau^\mathcal{I}$). Since $\tau^\mathcal{I}$ is maximal it
follows from Lemma \ref{nocermax ideals} that there are $F_i
\subseteq K_i$ for $i \in \omega$ such that $x \in
\overbar{\bigcup_{i \in \omega}F_i}$ and $\bigcup_{i \in
\omega}F_i$ is $\tau^\mathcal{I}$-open and hence
$(\mathcal{I},\tau^\mathcal{I})$-crowded. Find $\alpha \in
\mathfrak{c}$ such that $\langle A^\alpha_i : i \in \omega \rangle
= \langle F_i : i \in \omega \rangle$. Now we have that
$\bigcup_{i \in \omega}A^\alpha_i$ is
$(\mathcal{I},\tau_\alpha)$-crowded so Lemma \ref{q ideals} tells
us that there is an $(\mathcal{I},\tau_\alpha)$-crowded $S
\subseteq \bigcup_{i \in \omega}A^\alpha_i$ such that $|S \cap
A^\alpha_i| \leq 1$ for each $i \in \omega$ and $S$ is
$\tau_\alpha$-dense in $\bigcup_{i \in \omega}A^\alpha_i$. But now
condition 5 guarantees that $cl_{\tau_\alpha}(S) \in
\tau_{\alpha+1}$ and $cl_{\tau_\alpha}(S) \setminus S$ is
$\tau_{\alpha+1}$-discrete for some such $S$. Since
$\tau_{\alpha+1} \subseteq \tau^\mathcal{I}$, it follows that $x
\in cl_{\tau_\alpha}(\bigcup_{i \in \omega}A^\alpha_i) =
cl_{\tau_\alpha}(S) = \overbar{S}$.
\end{proof}

Next we want to apply this result to various specific ideals in
order to get maximal topologies satisfying all boolean
combinations of the properties $q^+$ and selective separability.
But first we need the following:

\begin{lemma}\label{denseopen}
If $(X,\tau)$ is a maximal $\mathcal{I}$-crowded space and $A \in
\mathcal{I}$ then $A$ is closed with empty interior (i.e. $X
\setminus A$ is dense open).
\end{lemma}

\begin{proof}
Since $\tau \cap \mathcal{I}= \{\emptyset\}$ we have that $A$ has
empty interior. Thus $X \setminus A$ is dense and therefore open
since $\tau$ is maximal.
\end{proof}

\begin{theorem}[$\mathfrak{m}_c=\mathfrak{c}$]\label{maxssq}
There exists a maximal zero-dimensional space which is selectively
separable and $q^+$.
\end{theorem}

\begin{proof}
Let $\mathcal{I}=[X]^{<\omega}$. It is well known and easy to see
that $\mathcal{I}$ is $p^+$, $q^+$ and HM so Theorem \ref{main}
gives us a maximal zero-dimensional topology $\tau^\mathcal{I}$ on
$X$ which is selectively separable and $q^+$.
\end{proof}

\begin{theorem}[$\mathfrak{m}_c=\mathfrak{c}$]\label{maxss-q}
There exists a maximal zero-dimensional space which is selectively
separable but not $q^+$.
\end{theorem}

\begin{proof}
Fix a partition of $X$ into finite sets $X=\bigcup_{n \in \omega}
F_n$ such that $|F_n|=n$. Let $\mathcal{I}=\bigcup_{m \in \omega}
\bigcap_{n \in \omega} \{A \subseteq X : |A \cap F_n| \leq m \}$.
Note that $\mathcal{I}$ is an $F_\sigma$ ideal and therefore it is
HM and $p^+$ so Theorem \ref{main} gives us a maximal
zero-dimensional topology $\tau^\mathcal{I}$ on $X$ which is
selectively separable. Now we show that $\tau^\mathcal{I}$ is not
$q^+$ at any point. Let $x \in X$ and note that $x \in
\overbar{\bigcup_{n \in \omega} G_n} \setminus \bigcup_{n \in
\omega} G_n$ where $G_n=F_n \setminus \{x\}$. However if $S
\subseteq \bigcup_{n \in \omega} G_n$ intersects each $G_n$ in at
most one point, we see that $S \in \mathcal{I}$ and by Lemma
\ref{denseopen} $S$ is closed. Thus $x \notin \overbar{S}$,
showing that $\tau^\mathcal{I}$ is not $q^+$ at $x$.
\end{proof}

\begin{theorem}[$\mathfrak{m}_c=\mathfrak{c}$]\label{maxq-ss}
There exists a maximal zero-dimensional space which is $q^+$ but
not selectively separable.
\end{theorem}

\begin{proof}
Fix a partition of $X$ into infinite sets $X=\bigcup_{n \in
\omega} X_n$ and let $\mathcal{I}$ be the collection of all
subsets of $X$ whose intersection with all but finitely many
$X_n$'s is finite. It is easy to see that $\mathcal{I}$ is $q^+$
and it is also HM being a Borel ideal. Theorem \ref{main} gives us
a maximal zero-dimensional topology $\tau^\mathcal{I}$ on $X$
which is $q^+$. Now we show that $\tau^\mathcal{I}$ is not
selectively separable. For each $i \in \omega$ let $D_i=
\bigcup_{n \geq i} X_n$. Note that each $X \setminus D_i \in
\mathcal{I}$ so by Lemma \ref{denseopen} it is dense. However if
$F_i \subseteq D_i$ is finite for each $i \in \omega$, we have
that $\bigcup_{i \in \omega} F_i \in \mathcal{I}$ and using again
Lemma \ref{denseopen} we get that $\bigcup_{i \in \omega} F_i$ is
closed and therefore not dense.
\end{proof}

Suppose $\mathfrak{m}_c=\mathfrak{c}$. Let $\mathcal{I}_1$ and
$\mathcal{I}_2$ be the ideals on $X$ defined as in the proofs of
theorems \ref{maxss-q} and \ref{maxq-ss} respectively. Now Theorem
\ref{main} gives us maximal zero-dimensional topologies
$\tau^{\mathcal{I}_1}$ and $\tau^{\mathcal{I}_2}$ on $X$ such that
$\tau^{\mathcal{I}_1}$ is not $q^+$ and $\tau^{\mathcal{I}_2}$ is
not selectively separable. Taking the disjoint union of these two
spaces we obtain a space which is neither SS nor $q^+$.  We show
next that this can be proved without assuming that
$\mathfrak{m}_c=\mathfrak{c}$.

\begin{theorem}\label{max-q-ss}
There exists a maximal zero-dimensional space which is neither
$q^+$ nor selectively separable.
\end{theorem}

\begin{proof}
Barman and Dow \cite{BarmanDow2011} have shown that there is a
maximal non SS space $Z$. We use the same idea in their proof to
show that there is a maximal non $q^+$ regular space. Let
$(X,\tau)$ be the space given in Example \ref{clopen}. Recall that
$X=\bigcup_k A_k$ where each $A_k$ is finite and every selector is
closed in $X$.  It was observed in \cite[Lemma
2.19]{BarmanDow2011} that van Douwen \cite{Vand} implicitly showed
that there is a regular topology $\tau'$ on $X$  finer than
$\tau$  and a dense subspace $Y$ of $(X,\tau')$ which is maximal.
Then $Y=\bigcup_k A_k\cap Y$ and this decomposition shows that $Y$
is not $q^+$. Finally, the disjoint union of $Z$ and $Y$ is the
required maximal space.
\end{proof}

It is consistent with ZFC that no countable maximal space is SS
\cite{BarmanDow2011,Reposvetal2010}. The same happens with the
$q^+$ property as we show next.

\begin{theorem}\label{nonq}
It is consistent that there are no maximal $q^+$ spaces.
\end{theorem}

\begin{proof}
Let $X$ be a countable maximal space and $x\in X$. Then
$\mathcal{U}_x=\{A\subseteq X:\; x\in \overline{A\setminus
\{x\}}\}$ is an ultrafilter (see \cite{Vand}). If $X$ is a $q^+$
space, then  $\mathcal{U}_x$ is a Q-point. Since it is consistent
that there are no Q-points (see \cite{Miller}),  we are done.
\end{proof}

\noindent {\bf Acknowledgments:}  The third author thanks La
Vicerrector\'ia de Investigaci\'on y Extensi\'on de la Universidad
Industrial de Santander for the financial support for this work,
which is part  of the VIE project  \#2422.

\bibliographystyle{plain}

\end{document}